\documentclass[12pt]{amsart}
\usepackage{hyperref}
\usepackage{amsfonts}
\usepackage{bbm}
\usepackage{esint}
\usepackage{mathrsfs}
\usepackage{amssymb,amsmath,amsfonts,amsthm,color}
\usepackage{setspace}
\numberwithin{equation}{section}

\DeclareMathOperator{\supp}{supp}

\allowdisplaybreaks

\theoremstyle{plain}
\newtheorem{theorem}{Theorem}[section]

\newtheorem{lemma}[theorem]{Lemma}

\theoremstyle{definition}
\newtheorem{definition}[theorem]{Definition}
\newtheorem{remark}[theorem]{Remark}
\newtheorem{proposition}[theorem]{Proposition}

\allowdisplaybreaks

\newcommand{\e}{\varepsilon}

\newcommand{\R}{\mathbb{R}}

\def\R{\mathbb R}

\def\12{\frac{1}{2}}

\def\b1 {\dot{B}_1^{-\alpha,1}}

\begin{document}

\setcounter{tocdepth}{3} \allowdisplaybreaks

\title[Variable Dunkl--Hardy spaces ]
{Dual spaces and T1 theorem of variable Dunkl--Hardy spaces
}

\author[J Tan]{Jian Tan}

\subjclass[2020]{Primary 42B35; Secondary 43A85, 42B25, 42B30}

\keywords{Dual space, Dunkl--Calder\'on--Zygmund operator, Littlewood--Paley theory, variable Dunkl--Hardy space, T1 Theorem}

\begin{abstract}
In this paper, we identify the dual spaces of variable Dunkl--Hardy spaces using a discrete Dunkl--Calder\'on reproducing formula and the Frazier and Jawerth's method adapted to the variable exponent setting. Furthermore, we prove a $T1$ theorem for Dunkl--Calder\'on--Zygmund operators acting on variable Dunkl--Hardy spaces and on their dual spaces via using almost orthogonal estimate, the duality result and the density argument.
    \end{abstract}
\maketitle

\section{Introduction}
The study of the classical real-variable Hardy spaces theory was initiated by Stein and Weiss \cite{SW} and systematically developed by Fefferman and Stein \cite{FS}.
Since then, the real Hardy spaces have played a fundamental role in harmonic analysis, partial differential equations and related areas. 

On the other hand, due to many applications to 
elasticity, fluid dynamics, calculus of variations, and differential equations with 
non-standard growth condition,
the variable exponent function spaces,
such as the variable Lebesgue spaces and the variable
Hady spaces, were studied by a substantial number
of researchers. For instance, see \cite{CF,DHHR2011} for variable Lebesgue spaces
and \cite{CW,Ho,NS,Tan2023,YZN} for variable Hardy spaces.
However, most of these developments are considered in Euclidean spaces or, more generally, in the spaces of homogeneous type. The study of variable Hardy spaces in the Dunkl settings is still quite limited. 

Dunkl Harmonic analysis, introduced by C. F. Dunkl in \cite{D2} in connection with finite reflection groups, provides a natural generalization of classical harmonic analysis by replacing partial derivatives with Dunkl operators. 
Recently, the theory of Dunkl Hardy spaces have been systematically studied in \cite{ABDH, ADH, DH2, JL,THHLL2}, including atomic decomposition, maximal characterizations, duality theory and boundedness of Dunkl singular integrals. Nevertheless, these results are mainly restricted to the constant exponent setting, and a complete theory of variable Hardy spaces in the Dunkl setting has not yet been established.

{\it In this paper, we aim to fill this gap by developing a duality theory and a $T1$ theorem for variable Hardy spaces in the Dunkl setting.}
More precisely, we show that the dual space of the variable Dunkl--Hardy space can be characterized by a suitable variable Dunkl--Carleson measure space. Moreover, we establish a $T1$ theorem for Calder\'on--Zygmund operators on variable Dunkl--Hardy spaces and variable Dunkl--Carleson measure space. The paper is organized as follows. Section 2 collects necessary preliminaries on Dunkl Harmonic analysis and variable function spaces. In Section 3, we establish the duality between variable Dunkl--Hardy spaces and variable Dunkl--Carleson measure spaces. Section 4 is devoted to the general $T1$ theorem for variable Dunkl--Hardy spaces and variable Dunkl--Carleson measure space.
The main tools in this paper include the Calder\'on reproducing formula in the Dunkl setting, discrete Dunkl--Littlewood--Paley theory, almost orthogonality estimates, and techniques from variable exponent analysis. The novelty of this paper lies in the intersection of Dunkl harmonic analysis and the theory of variable exponent function space, combining tools from both fields to address problems.

Throughout this paper, $C$ or $c$ denotes a positive constant that is independent of the main parameters involved but may vary at each occurrence. To denote the dependence of the constants on some parameter $r$, we will write $C{(r)}$. We denote $f\leq Cg$ by $f\lesssim g$. If $f\lesssim g\lesssim f$, we write $f\sim g$.

\section{Preliminaries and notation}

In this section, we give an account on results from the Dunkl theory, the theory of variable function spaces, the Dunkl--Calder\'on reproducing formula and   
the variable Dunkl--Hardy space which will be relevant for the sequel.

\subsection{Dunkl theory}
Let $R$ be a normalized root system in $\mathbb R^n$ so that $\langle\alpha, \alpha\rangle=2$ for $\alpha\in R$ and $R^+$ be a fixed positive subsystem, where 
$\langle \cdot, \cdot\rangle$ 
stands for the usual scalar product on $\mathbb R^n$. 
We denote by $G$ the finite reflection group generated by the reflections $\sigma_\alpha$ for all $\alpha\in R$, 
where $\sigma_\alpha(x)=x-\langle x,\alpha\rangle\alpha$
for $x\in \mathbb R^n$.
A multiplicity function $\kappa(\alpha)\ge 0$ is invariant under the natural action of $G$ defined on $R$, 
which will be fixed throughout this paper. For $\xi\in R^n$,
the Dunkl operators $T_\xi$ are defined by
$$T_\xi f(x)=\partial_\xi f(x) + \sum\limits_{\alpha\in R}\frac{\kappa(\alpha)}{2}\langle\alpha, \xi\rangle\frac{f(x)-f(\sigma_\alpha(x))}{\langle \alpha, x\rangle}.$$
We denote that $T_j=T_{e_j}$,
where $e_1, \cdots, e_n$ are the standard unit vectors of $\mathbb R^n.$
The Dunkl measure $dw(x)$ is defined by
\begin{eqnarray*}
dw(x)=\prod_{\alpha\in R}|\langle\alpha,x\rangle|^{\kappa(\alpha)}dx=\prod_{\alpha\in R^+}|\langle\alpha,x\rangle|^{2\kappa(\alpha)}dx
\end{eqnarray*}
be the associated measure in $\mathbb{R}^{n}$, where $dx$ stands for the Lebesgue
measure in $\mathbb{R}^{n}$. 
\noindent We denote by $L^p(\R^n,dw)$, $p\in[1,\infty]$, the space of measurable functions on
$\mathbb R^n$ such that
$$
\|f\|_{L^p(\R^n,dw)}=\left(\int_{\mathbb R^N}|f(x)|^pdw(x)\right)^{1/p}<\infty,\;1\le p<\infty,
$$
$$
\|f\|_{L^\infty(\R^n,dw)}=\mbox{ess}\sup_{x\in \mathbb R^n}|f(x)|<\infty.
$$
For fixed $y\in R^n$, the Dunkl kernel $x\rightarrow E(x,y)$ is a unique analytic solution to the system
$$T_\xi f=\big<\xi,y\big>f,\quad f(0)=1$$
for all $\xi\in\mathbb R^n$.

Dunkl transform, which enjoys properties
similar to the classical Fourier transform, is
defined by
$$ \widehat{f}(x)=c_h\int_{\mathbb{R}^{n}} E(x, -iy)f(y)h^2_{\kappa}(y)dy,$$
where the usual character $e^{-i\langle x,y\rangle}$ is replaced by
$E(x, -iy)=V_{\kappa}(e^{-i\langle \cdot,y\rangle})(x)$ for some
positive linear operator $V_{\kappa}$ and the weight functions
$h_{\kappa}$ are invariant under a finite reflection group $G$ on
$\mathbb{R}^{n}$. Particularly, the Dunkl transform also
satisfies the Plancherel identity, namely, $\|\widehat{f}\|_2=\|f\|_2$
and  if the parameter $\kappa=0,$ then $h_{\kappa}(x)=1$ and
$V_{\kappa}$ is the identity operator, thus the Dunkl transform reduces to the classical
Fourier transform automatically.  For more details, see \cite{D2}.

Dunkl translation is defined by
$$ {\widehat{\tau_y f}}(x)=E(y, -ix) \hat{f}(x)$$
for all $x\in \mathbb{R}^{n}.$ When the function $f$ is in the Schwartz class, the above equality holds pointwise. As an operator on $L^2(\mathbb R^n,h^2_{\kappa}),$ $\tau_y$ is bounded. However, it is not at all clear whether the translation operator can be defined for $L^p$ functions with
$p\not=2.$ For $f, g\in L^2(\mathbb{R}^{n},h^2_{\kappa}),$ their convolution
can be defined in terms of the translation operator by
$$f{\ast}_{\kappa} g(x)=\int_{\mathbb{R}^{n}}f(y)\tau_{x}\widetilde g(y)h^2_{\kappa}(y)dy,$$
where $\widetilde g(y)=g(-y).$ 
See \cite{TX1} for more information.

 There are two important metrics in the Dunkl setting which will be used in the whole paper. The Euclidean metric is defined by 
$$\|x-y\|:=\left\{\sum\limits_{j=1}^n|x_j-y_j|^2\right\}^\frac{1}{2}$$ and the so-called Dunkl metric is given by
$$d(x,y)=\min\limits_{\sigma\in G}\|x-\sigma(y)\|,$$ the distance  between two G-orbits $\mathcal{O}(x)$ and $\mathcal{O}(y).$ Obviously, 
$$d(x,y)\leqslant \|x-y\|, d(x,y)=d(y,x),\quad \mbox{and}\quad d(x,y)\leqslant
d(x,z)+d(z,y)$$ for all $x,y,z\in \R^n.$ Moreover,
$\omega(B(x,r))\sim \omega(B(y,r))$ when $d(x,y)\sim r,$ and
$$\omega(B(x,r))\leqslant \omega(B_d(x,r))\leqslant
|G|\omega(B(x,r)),$$ 
where $B(x,r):=\{y\in
\mathbb{R}^{n}:\|x-y\|<r\}$ and 
$$B_d(x,r):=\{y\in
\mathbb{R}^{n}:d(x,y)<r\}.$$
Moreover, 
for all $x\in\mathbb{R}^n$ and $0<r_1<r_2$, there exists a positive constant $C$ such that
\[
C^{-1}\left(\frac{r_2}{r_1}\right)^n
\le
\frac{\omega(B(x,r_2))}{\omega(B(x,r_1))}
\le
C\left(\frac{r_2}{r_1}\right)^N.
\]

This implies that
$d\omega(x)$ satisfies the doubling and reverse
doubling properties, that is, there is a constant $C > 0$ such that
for all $ x\in \mathbb{R}^{n}, r>0,\lambda \geqslant 1$ and ${N}=n+\sum\limits_{\alpha\in R}\kappa(\alpha),$
$$C^{-1}\lambda^{n}\omega(B(x, r)) \leqslant \omega(B(x, \lambda
r))\leqslant C\lambda^{ N}\omega(B(x, r)).$$ 
It turns out that the Dunkl setting, $(\Bbb R^n,
  \|\cdot\|, dw),$ is space of homogeneous type in the sense of
  Coifman and Wiess with the measure $dw$ satisfing the
  doubling and the reverse doubling properties. 

Denote $BMO(\mathbb{R}^n, d\omega)$ by the standard BMO space on $(\mathbb{R}^n,\|\cdot\|,d\omega)$. That is,
\[
BMO(\mathbb{R}^n,d\omega)
=
\left\{
b\in L_{\mathrm{loc}}^1(\mathbb{R}^n,d\omega):
\|b\|_*<\infty
\right\},
\]
where
\[
\|b\|_*
=
\sup_{B\subset\mathbb{R}^n}
\frac{1}{\omega(B)}
\int_B
|b(x)-b_B|\,d\omega(x)
<\infty
\]
with the supremum taken over all Euclidean balls
\[
B=B(y,r)
=
\left\{
z\in\mathbb{R}^n:
\|z-y\|<r
\right\},
\]
and
\[
b_B
=
\frac{1}{\omega(B)}
\int_B
b(x)\,d\omega(x).
\]

\subsection{Function spaces with variable exponents}
In $\mathbb R^n$,
for any variable exponent function $p(x)$ fulfilling
$$1\le p^-:=\mathrm{ess}\inf \{p(x):x\in\R^n\}\le p(x)\le p^+:=\mathrm{ess}\sup \{p(x):x\in\R^n\}<\infty,$$
the variable Lebesgue space $L^{p(\cdot)}(\mathbb R^n,dx)$ associated with Lebesgue measure $dx$ is defined to be the set of all measurable functions $f(x)$ defined on $\mathbb R^n$ such that
$$\int_{\mathbb R^n} |f(x)|^{p(x)}d x<\infty,$$
equipped with the Luxemburg--Nakano norm
$$\|f\|_{L^{p(\cdot)}(\mathbb R^n, dx)}
=\inf\left\{\lambda>0: \int_{\mathbb R^n} \left|\frac{f(x)}{\lambda}\right|^{p(x)}d x\le1 \right\}.$$
Obviously, when $p(x)=p$, the Lebesgue space with variable exponents $L^{p(\cdot)}(\mathbb R^n, dx)$ is reduced to the classical one $L^p(\mathbb R^n, dx)$.
As a special case of the theory of Nakano and Luxemberg,
we see that $L^{p(\cdot)}(\mathbb R^n, dx)$
is a quasi-normed space. Especially, when $p^-\geq1$, $L^{p(\cdot)}(\R^n, dx)$ is a Banach space.

To define the variable function spaces in the Dunkl setting, let us first give the definitions of variable exponent functions in the Dunkl setting.
In what follows, a measurable function $p(\cdot): \mathbb R\to (0,\infty)$ is said to be a variable exponent function.
Denote by $\mathcal{P}^0=\mathcal{P}^0(\mathbb R^n)$ the set of all variable exponents $p(\cdot)$ on $\mathbb R^n$ with $0< p^-\le p(x)\le  p^+<\infty$. 
Denote by $\mathcal{P}=\mathcal{P}(\R^n)$ the set of all variable exponents $p(\cdot)$ on $\R^n$ with $1< p^-\le p(x)\le  p^+<\infty$.
For any measurable functions $f(x)$ defined on $\R^n$ with respect to Dunkl measure $dw$, we also define the modular 
$$
\rho_{p(\cdot)}(f)=\int_{\R^n}|f(x)|^{p(x)}dw(x).
$$

\begin{definition}\cite{TT}
For any $p(\cdot)\in \mathcal{P}$, the Dunkl--Lebesgue space with variable exponents $L^{p(\cdot)}(\R^n,dw)$
is defined to be the set of all measurable functions $f(x)$ defined on $\R^n$ with respect to Dunkl measure $dw$ such that the Luxemburg-Nakano norms
$$\|f\|_{L^{p(\cdot)}(\R^n, dw)}=\inf\left\{\lambda>0:\rho_{p(\cdot)}\bigg(\frac{1}{\lambda}f(x)\bigg)\le1 \right\}<\infty.$$
\end{definition}

We assume throughout that the variable exponent $p(\cdot)$ is $G$-invariant, i.e., $p(\sigma x)=p(x)$ for all $\sigma\in G$ and $x\in\mathbb{R}^N$.
We also recall some remarkable results on Lebesgue space with variable exponents in \cite{HHP}, which can be applied to the Dunkl setting and will be used in our proofs of the main results.

By $M_Df$ we denote the Dunkl--Hardy--Littlewood maximal function of $f$, i.e.
$$
M_Df(x):=\sup_{t>0}\frac{1}{w(B(x,t))}\int_{B(x,t)}|f(y)|dw(y).
$$
We denote that $\widetilde M_Df(x):=\sum_{\sigma\in G}M_Df(\sigma(x))$.
By $\mathcal B$ we denote the set of bounded exponents $p(x)\in \mathcal P$ such that the function $M_D$ is bounded on $L^{p(\cdot)}(\R^n, dw)$.
An important subset of $\mathcal{B}$ is the well-known $LH$
condition.
A variable exponent function $p(\cdot) \in \mathcal{P}$ is said to be locally log-H\"{o}lder continuous in $\mathbb{R}^n$,
if there exists a constant $c_{\log }(p)>0$ such that for any $x, y \in {\R^n}$,
$$
|p(x)-p(y)| \leq \frac{c_{\log }(p)}{\log (\e+1/{\|x-y\|})},
$$
and is said to satisfy the log-H\"{o}lder decay condition if
there exist a $p_\infty\in\R$ and a constant $c_\infty(p)>0$ such that for any $x\in\R^n$,
$$
|p(x)-p_\infty| \leq \frac{c_\infty(p)}{\log (e+\|x\|)}.
$$
A variable exponent function $p(\cdot) \in \mathcal{P}$ is said to satisfy the globally log-H\"{o}lder continuous condition,
 denoted by $p(\cdot)\in LH$,
 if $p(\cdot)$ satisfies both the locally log-H\"{o}lder continuous condition and the log-H\"{o}lder decay condition.
We remark that, for the unbounded (quasi)metric measure space case, the above $LH$ definitions were first proved in \cite[Lemma 2.1]{AHH2015}.
Thus, as a consequence of \cite[Corollary 1.8]{AHH2015}, if the variable exponent function $p(\cdot)\in LH\cap \mathcal P$,
then the Dunkl--Hardy--Littlewood maximal function $M_Df$ of $f\in L^{p(\cdot)}(\R^n, dw)$
is bounded on $L^{p(\cdot)}(\R^n)$.
By using \cite[Theorem 1.1]{K}, if $p(\cdot)\in\mathcal B$, then the Dunkl--Hardy--Littlewood maximal function $M_Df$
is also bounded on $L^{p'(\cdot)}(\R^n)$, where $\frac{1}{p(x)}+\frac{1}{p'(x)}=1$ for any $x\in\mathbb R^n$, i. e. $p'(\cdot)\in \mathcal B$. 
The known Fefferman--Stein vector-valued inequality on spaces of homogeneous type $X=(\mathbb R^n, dw)$ in \cite[Theorem 2.7]{ZSY16} is also needed in our proofs.
\begin{lemma}\label{FS}
Let $p(\cdot)\in \mathcal P\cap LH$ and $u\in(1,\infty)$. Then for any measurable functions sequence $\{f_i\}_{i=1}^\infty\subset L^{p(\cdot)}(\mathbb R^n, dw)$,
\begin{equation*}
\left\|\left\{\sum_{i=1}^\infty[M_D(f_i)]^u\right\}^{\frac{1}{u}}\right\|_{L^{p(\cdot)}(\mathbb R^n, dw)}\le
C\left\|\left(\sum_{i=1}^\infty|f_i|^u\right)^{\frac{1}{u}}\right\|_{L^{p(\cdot)}(\mathbb R^n, dw)}.
\end{equation*}
\end{lemma}





\subsection{Dunkl--Calder\'on reproducing formula and   
Littlewood--Paley theory}
The Calder\'on reproducing formula in the Dunkl setting plays key role in the whole paper.
Now we recall the discrete Dunkl--Calder\'on reproducing formula on $L^2(\mathbb R^n)$ established in \cite{THHLL1}. 
Applying the following Calder\'on reproducing formula given in \cite{ADH},

\begin{equation}\label{ccrf}
f(x)=\int_0^\infty {\psi_{t}}\ast {q_{t}}\ast
f(x)\frac{dt}{t},
\end{equation}
where the integral converges in $L^2(\mathbb R^n)$ and ${q_t}f=t\frac{\partial}{\partial t}p_tf$ with $p_t$ is
the Dunkl--Poisson kernel, and ${\psi}\in C^\infty_0(B_d(0,1/4))$
is a radial function with $\int_{\mathbb R^nN}{\psi}(x)d\omega(x)=0$.
the authors in \cite{THHLL1} decompose $f$ by the following

$$f(x)=\int_0^\infty {\psi}_t\ast {q_t}\ast f(x)\frac{dt}{t}=T_M(f)(x) + R_1(f)(x) + R_M(f)(x),$$
where
$$T_M(f)(x)=-\ln r \cdot\sum\limits_{j=-\infty}^\infty\sum\limits_{Q\in {\mathcal Q}^j}w(Q) \psi_{j}(x,x_{Q})q_{j}\ast
f(x_{Q}),$$
$$R_1(f)(x)=-\sum\limits_{j=-\infty}^\infty\int_{r^{-j}}^{r^{-j+1}}
\Big[{\psi_t}\ast {q_t}\ast f(x)- \psi_{j}\ast q_{j}\ast
f(x)\Big]\frac{dt}{t}$$ and
$$R_M(f)(x)=-\ln r \cdot\sum\limits_{j=-\infty}^\infty\sum\limits_{Q\in {\mathcal Q}^j}\int_{Q}
\Big[\psi_{j}(x,y)q_{j}\ast f(y)-\psi_{j}(x,x_{Q})q_{j}\ast
f(x_{Q})\Big]d\omega(y),$$ where $\psi_j={\psi}_{r^j},
q_{j}={q}_{r^j}$ with $1<r\leqslant r_0$ for some fixed $r_0,
{\mathcal Q}^j$ is the collection of all ``$r$-dyadic cubes'' $Q$
with the side length $r^{-M-j}$ for $M$ is some fixed large integer,
and $x_{Q}$ is any fixed point in the cube $Q.$
By applying Coifman's decomposition of the identity on $L^2(\mathbb R^n)$
gives
$$I=T_M + R_1 + R_M.$$
The authors in \cite{THHLL1} showed that $R_1$ and $R_M$ are bounded on $L^p(\mathbb R^n, d\omega), 1<p<\infty,$ and moreover, $\|R_1+R_M\|_{p,p}< 1$ with  some a fixed $r_0>1$ and a fixed $M.$ This implies that $(T_M)^{-1},$ the inverse
of $T_M,$ exists and is bounded on $L^p(\mathbb R^n, d\omega), 1<p<\infty.$
Then the discrete type Dunkl--Calder\'on reproducing formula is given by the following theorem.
Note that the well-known discrete type Calder\'on reproducing formula was
first introduced by Frazier and Jawerth in \cite{FJ}.

\begin{theorem}[\cite{THHLL1}]
Let $p(\cdot)\in LH$ and $\frac{N}{N+1}<p^-\le p^+\le 1$. 
If $f\in L^2(\mathbb R^n,d\omega)$,
 then there exists a function $h:=(T_M)^{-1}f\in L^2(\mathbb R^n, d\omega)$ 
 such that $\|f\|_{L^2}\sim \|h\|_{L^2}$ and 
    \begin{align*}
    f(x)=\sum\limits_{j=-\infty}^\infty\sum\limits_{Q\in {\mathcal Q}^j}\omega(Q)
    \psi_{Q}(x,x_{Q})q_{Q}
    h(x_{Q}),
    \end{align*}
    where the series converges in $L^2(\mathbb R^n,d\omega)$ with $\psi_Q=\psi_{j}, q_{Q}=q_{j}$ when $Q\in {\mathcal Q}^j$ where $\{{\mathcal Q}^j\}$ is the collection of all ``r-dyadic cubes'' $Q$
     with the side length $r^{-M-j}$ for $M$ is some fixed large integer and $x_{Q}$ is any fixed point in the cube $Q.$
\end{theorem}

This discrete type wavelet decomposition leads to the following generalized discrete Littlewood--Paley function in  \cite{THHLL1}.
\begin{definition}\label{df1.2}
    For $f\in L^2(\mathbb R^n, d\omega), S(f),$ the \emph{discrete Littlewood--Paley square function} of $f,$ is defined by
    \begin{eqnarray}\label{square_function}
    S(f)(x):= \left\{
    \sum\limits_{j=-\infty}^\infty\sum\limits_{Q\in {\mathcal Q}^j}|q_{Q}f(x_{Q})|^2\chi_{Q}
    (x) \right\}^{1/2},
    \end{eqnarray}
    where $\chi_{Q}(x)$ is the characteristic function of the cube $Q.$
\end{definition}

We also recall the Littlewood--Paley theory on space of homogeneous type $(\mathbb R^n, \| \cdot\|, d\omega)$ in the sense of Coifman and Weiss. We begin with the following definition of the test functions in space of homogeneous type $(\mathbb R^n,\|\cdot\|,d\omega):$

    \begin{definition}\label{test} A function $f(x)$ defined on $\mathbb R^n$ is said to be a test function if there exits a constant $C$ such that for $0<\beta\leqslant 1, \gamma>0, r>0$ and $x_0\in \mathbb R^n,$
        \begin{enumerate}
            \item[(i)] $|f(x)|\leqslant \frac C{V(x, r+\|x-x_0\|)}\big(\frac{r}{r+\|x-x_0\|}\big)^\gamma;$
            \item[(ii)] $\displaystyle |f(x)-f(x')|\leqslant C\big(\frac{\|x-x'\|}{r+\|x-x_0\|}\big)^\beta\frac{1}{V(x,r+\|x-x_0\|)}\big(\frac{r}{r+\|x-x_0\|}\big)^\gamma, \\ \quad for \|x-x'\|\leqslant \frac{1}{2}(r+\|x-x_0\|);$
            \item[(iii)] $\displaystyle \int_{\R^N} f(x)d\omega(x)=0.$
        \end{enumerate}

        We denote such a test function by $f\in \mathcal M(\beta,\gamma,r,x_0)$ and $\|f\|_{\mathcal M(\beta,\gamma,r,x_0)}$, the norm in $\mathcal M(\beta,\gamma,r,x_0),$ is defined by the smallest $C$
        satisfying the above conditions (i) and (ii).
    \end{definition}

    Applying Coifman's approximation to the identity and decomposition of the identity operator together with the Calder\'on--Zygmund operator theory,
    the discrete Calder\'on reproducing formula in space of homogeneous type is given by the following

    \begin{theorem}\label{dCRh}
        Let $\{S_k\}_{k\in \Bbb Z}$ be a Coifman's approximation to the identity and set
        ${D}_k =: {S}_k - {S}_{k-1}$. Then there exist two families of operators
        $\{\widetilde{D}_k\}_{k\in \Bbb Z}$ and $\widetilde{\widetilde{D}}_k$ such that for any fixed $x_{Q}\in Q$ with $Q\in {\mathcal Q}^k$ with $k\in \Bbb Z$ and ${\mathcal Q}^k$ are "$r-$ dyadic cubes"
        with the side length $r^{-M-k},$
        $$f(x)=\sum\limits_{k=-\infty}^\infty\sum\limits_{Q\in {\mathcal Q}^k}\omega(Q){\widetilde D}_k(x,x_{Q})D_k(f)(x_{Q})=\sum\limits_{k=-\infty}^\infty\sum\limits_{Q\in {\mathcal Q}^k}\omega(Q)D_k(x,x_{Q})
        \widetilde{\widetilde{D}}_k(f)(x_{Q}),$$
        where the series converge in $L^p(\omega), 1<p<\infty, \mathcal M(\beta,\gamma,r,x_0),$ and in $(\mathcal M(\beta,\gamma,r,x_0))^\prime,$ the dual of in $\mathcal M(\beta,\gamma,r,x_0).$
        Moreover, the kernels of the operators $\widetilde{D}_k$
        satisfy the following conditions: for $0<\varepsilon\leqslant 1,$
        \begin{enumerate}
            \item[(i)] $|\widetilde{D}_k(x,y)|\leqslant C\frac1{V_k(x)+V_k(y)+V(x,y)}\Big(\frac{r^{-k}}{r^{-k}+\|x-y\|}\Big)^\varepsilon;$
            \item[(ii)] $|\widetilde{D}_k(x,y)-\widetilde{D}_k(x',y)|\leqslant C\Big(\frac{\|x-x'\|}{r^{-k}+\|x-x'\|}\Big)^\varepsilon\frac1{V_k(x)+V_k(y)+V(x,y)}\Big(\frac{r^{-k}}{r^{-k}+\|x-y\|}\Big)^\varepsilon$,
            \item[]for $\|x-x'\|\leqslant (r^{-k}+\|x-y\|)/2;$
            \item[(iii)] $\displaystyle \int_{\Bbb R^N} \widetilde{D}_k(x,y)d\omega(x)=0\qquad \text{for all}\ y\in \Bbb R^n;$
            \item[(iv)] $\displaystyle \int_{\Bbb R^N} \widetilde{D}_k(x,y)d\omega(y)=0\qquad \text{for all}\ x\in \Bbb R^n.$
            \end{enumerate}
And the kernels of $\widetilde{\widetilde{D}}_k(x,y)$ satisfy the
conditions \rm{(iii)--(iv)} and \rm{(i)--(ii)} with $x$ and $y$
interchanged.
    \end{theorem}

\section {Variable Dunkl--Carleson measure space}

In this section, we will define the variable Dunkl--Carleson measure spaces
$\mathcal C^{p(\cdot)}_D(\mathbb R^n)$ in terms of discrete Littlewood--Paley theory.
Then we show that $\mathcal C^{p(\cdot)}_D(\mathbb R^n)$ is the dual space of the 
variable Dunkl--Hardy space $H_D^{p(\cdot)}(\mathbb R^n)$ via using
the duality inequality and the sequence spaces in the variable 
setting.
Before we state the definition of variable Dunkl--Carleson measure space, we first recall
the definition of the variable Dunkl--Hardy space in \cite{Tan26-1}.


\begin{definition}\label{df1.3}
    For $p(\cdot)\in \mathcal P^0$,  the variable Dunkl--Hardy space $H_D^{p(\cdot)}{(\mathbb R^n)}$ is defined as the completion of
    $$
    \left\{f\in L^2(\mathbb R^n): S(f)\in L^{p(\cdot)}(\mathbb R^n, d\omega)\right\}
    $$
     under the quasi-norm
$\|f\|_{H_D^{p(\cdot)}{(\mathbb R^n)}}:=\|S(f)\|_{L^{p(\cdot)}}(\mathbb R^n, d\omega)$.
\end{definition}

If $p(\cdot)\in \mathcal P\cap LH$, then the following lemma yields that $H_D^{p(\cdot)}{(\mathbb R^n)} \cong  L^{p(\cdot)}(\mathbb R^n, d\omega)$
with equivalent norms. 
The Littlewood--Paley characterization for variable Dunkl--Lebesgue spaces $L^{p(\cdot)}(\mathbb R^n, d\omega)$ is also given in \cite{Tan26-1}.
For more details on variable Hardy spaces, see \cite{S, Tan2, Tan22, YZZ}.

Then the discrete Calder\'on reproducing formula for $f\in L^2(\mathbb R^n,d\omega)$
with respect to the variable Dunkl--Hardy norms are given by the following

\begin{theorem}[\cite{Tan26-1}]\label{th1.1}
Let $p(\cdot)\in LH$ and $\frac{N}{N+1}<p^-\le p^+\le 1$. 
If $f\in L^2(\mathbb R^n,d\omega)\cap H_D^{p(\cdot)}(\mathbb R^n, d\omega)$,
 then there exists a function $h:=(T_M)^{-1}f\in L^2(\mathbb R^n, d\omega)\cap H_D^{p(\cdot)}(\mathbb R^n)$ 
 such that $\|f\|_{L^2}\sim \|h\|_{L^2}$ with $\|h\|_{H_D^{p(\cdot)}}\sim \|f\|_{H_D^{p(\cdot)}}$ and 
    \begin{align*}
    f(x)=\sum\limits_{j=-\infty}^\infty\sum\limits_{Q\in {\mathcal Q}^j}\omega(Q)
    \psi_{Q}(x,x_{Q})q_{Q}
    h(x_{Q}),
    \end{align*}
    where the series converges in $L^2(\mathbb R^n,d\omega)$ and $H_D^{p(\cdot)}{(\mathbb R^n)}$ with $\psi_Q=\psi_{j}, q_{Q}=q_{j}$ when $Q\in {\mathcal Q}^j$ where $\{{\mathcal Q}^j\}$ is the collection of all ``r-dyadic cubes'' $Q$
     with the side length $r^{-M-j}$ for $M$ is some fixed large integer and $x_{Q}$ is any fixed point in the cube $Q.$
\end{theorem}

Moreover, the discrete Calder\'on reproducing formula in Theorem \ref{dCRh} leads the
    following discrete square function on space of homogeneous type
    $(\Bbb R^n, \|\cdot\|, d\omega).$
    \begin{definition}\label{scw}
        Suppose that $f\in (\mathcal M(\beta,\gamma,r,x_0))^\prime.$ $S_{cw}(f),$ the Littlewood--Paley square function of $f$ on space of homogeneous type $(\mathbb R^n, \|\cdot\|, d\omega),$ is defined by
        $$S_{cw}(f)(x)=\bigg\{\sum\limits_{k=-\infty}^\infty\sum\limits_{Q\in {\mathcal Q}^k}
        |D_kf(x_{Q})|^2\chi_{Q}(x)\bigg\}^{1/2},$$
        where $D_k, {\mathcal Q}^k$ are same as given in Theorem \ref{dCRh}. See \cite{HMY} for more details.
    \end{definition}
\begin{definition}\label{hpcw}
For $p(\cdot)\in\mathcal P^0$,  the space $H^{p(\cdot)}_{cw}(\mathbb R^n,d\omega)$ is the collection of all distributions $f\in (\mathcal M(\beta,\gamma,r,x_0))^\prime$ such that
    $\|f\|_{H^{p(\cdot)}_{cw}(\mathbb R^n,d\omega)}=\|S_{cw}(f)\|_{L^{p(\cdot)}(\mathbb R^n,dw)}<\infty,$ where the square function $S_{cw}(f)$ is defined in the Definition \ref{scw}.
\end{definition}
The relationship between the Dunkl--Hardy space $H_D^{p(\cdot)}(\mathbb R^n)$ and Hardy space $H^{p(\cdot)}_{cw}(\mathbb R^n,d\omega)$ is given by the following

\begin{theorem}\cite{Tan26-1}\label{relation}
Let $p(\cdot)\in\mathcal P^0\cap LH$ with $\frac{{ N}}{{ N}+1}<p^-\le p^+<\infty.$ Then for any $f\in H_D^{p(\cdot)}{(\mathbb R^n)}$ we have
$$\|f\|_{H_D^{p(\cdot)}{(\mathbb R^n)}}\sim \|f\|_{H^{p(\cdot)}_{cw}(\mathbb R^n,d\omega)}.$$
\end{theorem}    

Now we define variable Dunkl--Carleson measure space as follows.
\begin{definition}\label{cmo}
    For $p(\cdot)\in \mathcal P^0$,  the variable Dunkl--Carleson measure space 
    $\mathcal C^{p(\cdot)}_D(\mathbb R^n)$ is defined as the completion of
    $$
    \left\{f\in L^2(\mathbb R^n, d\omega): \|f\|_{\mathcal C^{p(\cdot)}_D}<\infty\right\}.
    $$
     The quasi-norm of $\mathcal C^{p(\cdot)}_D(\mathbb R^n)$ is defined by
 $$\|f\|_{\mathcal C^{p(\cdot)}_D} := \sup\limits _{P}
    \left\{ {\omega(P)\over\|\chi_P\|^2_{L^{p(\cdot)}}}
    \sum\limits_{ Q \subseteq P }\omega(Q)
     \big|\psi_{Q}f (x_Q)
    \big|^2 \right\}^{1/2},$$
    where $P$ runs over all dyadic cubes, $\psi$ is given in Theorem \ref{th1.1}, $\psi_{Q}=\psi_{j}$ when $Q\in {\mathcal Q}^j$ and $x_Q$ are any fixed points in $Q$.
\end{definition}

The following duality inequality plays a crucial role for
characterizing the dual spaces of variable Dunkl--Hardy spaces.

\begin{proposition}\label{holder}
  Let $p(\cdot)\in LH$ with $\frac{{N}}{{N}+1}<p^-\le p^+\le 1$.
    For $f, g\in L^2{(\mathbb R^n, d\omega)}$, then
    
    $$|\langle f, g\rangle|\leqslant C \|f\|_{H_D^{p(\cdot)}}\|g\|_{\mathcal C_D^{p(\cdot)}}.$$
   
\end{proposition}

We first show Proposition \ref{holder}.
\begin{proof}[{\bf The proof of Proposition \ref{holder}}]
    By Theorem \ref{th1.1}, we know that, 
    for $f\in L^2(\mathbb R^n, d\omega)\cap H_D^{p(\cdot)}(\mathbb R^n, d\omega)$,
    we write
    $$f=\sum\limits_{j=-\infty}^\infty\sum\limits_{Q\in {\mathcal Q}^j}w(Q)
    \psi_{Q}(x,x_{Q})q_{Q}h(x_{Q}),$$
    where $\|h\|_{L^2}\sim \|f\|_{L^2}$ and 
    $\|h\|_{H_D^{p(\cdot)}}\sim \|f\|_{H_D^{p(\cdot)}}.$

    Denote that $$\Omega_{\ell}=\left\{x\in \mathbb{R}^{n}: S(h)(x)>2^{\ell}\right\}$$ and
    $$B_{\ell}=\Big\{ Q: Q \ \ \text{are $r$-dyadic cubes}, \ \  \omega\big(Q\cap \Omega_{\ell}\big)>\frac{1}{2}\omega(Q) \ \ \text{and} \ \ \omega\big(Q\cap \Omega_{\ell+1}\big)\leqslant\frac{1}{2}\omega(Q)\Big\}.$$
Also denote that ${{\widetilde \Omega}_\ell }=\{x\in \mathbb{R}^N: M_D(\chi_{\Omega_\ell})(x)>\frac{1}{2}\},$ where $M_D$ is the Hardy--Littlewood maximal function on $\mathbb{R}^n$
    with the measure $d\omega$.

    Denote $B^*_\ell:=\{ {\mathcal Q}^*_\ell\}_\ell$ by the maximal $r-$dyadic cubes in $B_\ell$ for $\ell\in \mathbb{Z}$.
     By the stopping-time construction, We get that
    \begin{eqnarray*}\label{discrete}
    f=
        \sum\limits_{\ell=-\infty}^\infty
        \sum\limits_{{\mathcal Q}^*_\ell\in B^*_\ell}
        \sum\limits_{Q\subseteq {\mathcal Q}^*_\ell \atop Q\in B_\ell}\omega(Q)\psi_{Q}(x,x_{Q})q_{Q}h(x_{Q}).
    \end{eqnarray*}

  Then, by the H\"older inequality we conclude that for $g\in L^2(\mathbb R^n,d\omega),$
     \begin{align*}
    &|\langle f,g\rangle|=
    \left|\sum\limits_{\ell=-\infty}^\infty\sum\limits_{{\mathcal Q}^*_\ell\in B^*_\ell}
    \sum\limits_{Q\subseteq {\mathcal Q}^*_\ell \atop Q\in B_\ell}w(Q)q_Qh(x_Q)\psi_Qg(x_Q)\right|\\
    \lesssim& \sum\limits_{\ell=-\infty}^\infty
    \sum\limits_{{\mathcal Q}^*_\ell\in B^*_\ell}\Big(\sum\limits_{Q\subseteq {\mathcal Q}^*_\ell \atop Q\in B_\ell}w(Q)|q_Qh(x_Q)|^2\Big)^{\frac{1}{2}}
    \Big(\sum\limits_{Q\subseteq {\mathcal Q}^*_\ell \atop Q\in B_\ell}w(Q)|\psi_Qg(x_Q)|^2\Big)^{\frac{1}{2}}\\
   =& \sum\limits_{\ell=-\infty}^\infty
    \sum\limits_{{\mathcal Q}^*_\ell\in B^*_\ell}
    \frac{\|\chi_{\mathcal Q_\ell^\ast}\|_{L^{p(\cdot)}}}{\omega(\mathcal Q_\ell^\ast)^{\frac{1}{2}}}
    \Big(\sum\limits_{Q\subseteq {\mathcal Q}^*_\ell \atop Q\in B_\ell}w(Q)|q_Qh(x_Q)|^2\Big)^{\frac{1}{2}}
    {\omega(\mathcal Q_\ell^\ast)^{\frac{1}{2}}\over\|\chi_{\mathcal Q_\ell^\ast}\|_{L^{p(\cdot)}}}
    \Big(\sum\limits_{Q\subseteq {\mathcal Q}^*_\ell \atop Q\in B_\ell}w(Q)|\psi_Qg(x_Q)|^2\Big)^{\frac{1}{2}}\\  
  \lesssim& \|g\|_{\mathcal C^{p(\cdot)}_D}\sum\limits_{\ell=-\infty}^\infty
    \sum\limits_{{\mathcal Q}^*_\ell\in B^*_\ell}
    \frac{\|\chi_{\mathcal Q_\ell^\ast}\|_{L^{p(\cdot)}}}{\omega(\mathcal Q_\ell^\ast)^{\frac{1}{2}}}
    \Big(\sum\limits_{Q\subseteq {\mathcal Q}^*_\ell \atop Q\in B_\ell}w(Q)|q_Qh(x_Q)|^2\Big)^{\frac{1}{2}}
   \\
    =&: \|g\|_{\mathcal C^{p(\cdot)}_D}\sum\limits_{\ell=-\infty}^\infty
    \sum\limits_{{\mathcal Q}^*_\ell\in B^*_\ell}\lambda_{\ell,\mathcal Q_\ell^\ast}.
    \end{align*}

    Notice that $p^+\le 1$, we obtain that
       \begin{align*}
&\sum\limits_{\ell=-\infty}^\infty
    \sum\limits_{{\mathcal Q}^*_\ell\in B^*_\ell}\lambda_{\ell,\mathcal Q_\ell^\ast}
    =\sum\limits_{\ell=-\infty}^\infty
    \sum\limits_{{\mathcal Q}^*_\ell\in B^*_\ell}\left\|\lambda_{\ell,\mathcal Q_\ell^\ast}\chi_{{\mathcal Q}^*_\ell}\over\|\chi_{{\mathcal Q}^*_\ell}\|_{L^{p(\cdot)}}\right\|_{L^{p(\cdot)}}\\
    &\le \left\|\sum\limits_{\ell=-\infty}^\infty
    \sum\limits_{{\mathcal Q}^*_\ell\in B^*_\ell}{\lambda_{\ell,\mathcal Q_\ell^\ast}\chi_{{\mathcal Q}^*_\ell}\over\|\chi_{{\mathcal Q}^*_\ell}\|_{L^{p(\cdot)}}}\right\|_{L^{p(\cdot)}}\\
        &\le \left\|\left(\sum\limits_{\ell=-\infty}^\infty
    \sum\limits_{{\mathcal Q}^*_\ell\in B^*_\ell}
    \left({\lambda_{\ell,\mathcal Q_\ell^\ast}\chi_{{\mathcal Q}^*_\ell}\over\|\chi_{{\mathcal Q}^*_\ell}\|_{L^{p(\cdot)}}}\right)^{p_-}\right)^{\frac{1}{p_-}}\right\|_{L^{p(\cdot)}}\\
        &= \left\|\left(\sum\limits_{\ell=-\infty}^\infty
    \sum\limits_{{\mathcal Q}^*_\ell\in B^*_\ell}
    \left({  \omega(\mathcal Q_\ell^\ast)^{-\frac{1}{2}}
    \Big(\sum\limits_{Q\subseteq {\mathcal Q}^*_\ell \atop Q\in B_\ell}w(Q)|q_Qh(x_Q)|^2\Big)^{\frac{1}{2}}
    \chi_{{\mathcal Q}^*_\ell}}\right)^{p_-}\right)^{\frac{1}{p_-}}\right\|_{L^{p(\cdot)}}.
    \end{align*}  
    
  Observe that $$\sum\limits_{Q\subseteq {\mathcal Q}^*_\ell \atop Q\in B_\ell}w(Q)|q_Qh(x_Q)|^2\lesssim 2^{2\ell}\omega(\mathcal Q_\ell^\ast).$$
  In fact, first for any $x\in Q$ with $Q\in B_\ell$, we get that
   $$w\big(\big({\widetilde \Omega}_\ell/\Omega_{\ell+1}\big)\cap Q\big)= w(Q)-w(\Omega_{\ell+1}\cap Q) \geqslant \frac{1}{2}w(Q)$$
   and 
   $$
   M_D(\chi_{Q\cap \tilde\Omega_\ell\setminus \Omega_{\ell+1}})(x)>\frac{1}{2}.
   $$
  Then by the classical Fefferman--Stein vector valued inequality,
  we have
  \begin{align*}
 & \sum\limits_{Q\subseteq {\mathcal Q}^*_\ell \atop Q\in B_\ell}w(Q)|q_Qh(x_Q)|^2\\
  &\le \int_{\mathbb R^n}\sum\limits_{Q\subseteq {\mathcal Q}^*_\ell \atop Q\in B_\ell}|q_Qh(x_Q)|^2\chi_Q(x)dx\\
   &\lesssim \int_{\mathbb R^n}\sum\limits_{Q\subseteq {\mathcal Q}^*_\ell \atop Q\in B_\ell}|q_Qh(x_Q)|M_D(\chi_{Q\cap \tilde\Omega_\ell\setminus \Omega_{\ell+1}})(x)|^2dx\\ 
      &\lesssim \int_{\mathbb R^n}\sum\limits_{Q\subseteq {\mathcal Q}^*_\ell \atop Q\in B_\ell}|q_Qh(x_Q)|^2\chi_{Q\cap \tilde\Omega_\ell\setminus \Omega_{\ell+1}}(x)dx\\ 
      &\lesssim \int_{{\mathcal Q_\ell^\ast\cap \tilde\Omega_\ell\setminus \Omega_{\ell+1}}}\sum\limits_{Q\subseteq {\mathcal Q}^*_\ell \atop Q\in B_\ell}|q_Qh(x_Q)|^2\chi_{Q\cap \tilde\Omega_\ell\setminus \Omega_{\ell+1}}(x)dx\\ 
      &\lesssim 2^{2\ell}\omega(\mathcal Q_\ell^\ast)
  \end{align*} 
  
 Therefore, we have
 
        \begin{align*}
\sum\limits_{\ell=-\infty}^\infty
    \sum\limits_{{\mathcal Q}^*_\ell\in B^*_\ell}\lambda_{\ell,\mathcal Q_\ell^\ast}
        &\lesssim \left\|\left(\sum\limits_{\ell=-\infty}^\infty
    \sum\limits_{{\mathcal Q}^*_\ell\in B^*_\ell}
    \left({  \omega(\mathcal Q_\ell^\ast)^{-\frac{1}{2}}
    \Big( 2^{2\ell}\omega(\mathcal Q_\ell^\ast)\Big)^{\frac{1}{2}}
    \chi_{{\mathcal Q}^*_\ell}}\right)^{p_-}\right)^{\frac{1}{p_-}}\right\|_{L^{p(\cdot)}}\\
    &= \left\|\left(\sum\limits_{\ell=-\infty}^\infty
    \sum\limits_{{\mathcal Q}^*_\ell\in B^*_\ell}
    \left({ 
    2^{\ell}
    \chi_{{\mathcal Q}^*_\ell}}\right)^{p_-}\right)^{\frac{1}{p_-}}\right\|_{L^{p(\cdot)}}\\
     &\lesssim \left\|\left(\sum\limits_{\ell=-\infty}^\infty
    \left( 
    2^{\ell}
    \chi_{\widetilde {\Omega_\ell}}\right)^{p_-}\right)^{\frac{1}{p_-}}\right\|_{L^{p(\cdot)}}\\
         &\lesssim \left\|\left(\sum\limits_{\ell=-\infty}^\infty
    \left( 
    2^{\ell}
    \chi_{{\Omega_\ell}}\right)^{p_-}\right)^{\frac{1}{p_-}}\right\|_{L^{p(\cdot)}}\\
         &\lesssim \left\|\left(\sum\limits_{\ell=-\infty}^\infty
    \left( 
    2^{\ell}
    \chi_{{\Omega_\ell}\setminus\Omega_{\ell+1}}\right)^{p_-}\right)^{\frac{1}{p_-}}\right\|_{L^{p(\cdot)}}\\
    &\lesssim \left\|\left(\sum\limits_{\ell=-\infty}^\infty
    \left( 
   S(h)
    \chi_{ {\Omega_\ell}\setminus\Omega_{\ell+1}}\right)^{p_-}\right)^{\frac{1}{p_-}}\right\|_{L^{p(\cdot)}}\\
    &=\|S(h)\|_{L^{p(\cdot)}}=\|h\|_{H_D^{p(\cdot)}}=\|f\|_{H_D^{p(\cdot)}},
    \end{align*} 
    where the second inequality follows from the fact that
    $\bigcup\limits_{{\mathcal Q}^*_\ell\in
        B^*_\ell}{\mathcal Q}^*_\ell \subseteq \widetilde{\Omega}_\ell$,
        the third inequality follows from the classical Feffereman--Stein vector valued inequality
         and the forth inequality follows from the fact that 
         $$
         \sum_{\ell}2^\ell\chi_{\Omega_\ell}(x)=2\sum_{\ell}2^\ell \chi_{\Omega_\ell\setminus\Omega_{\ell+1}}.
         $$
         
Hence, we get
    \begin{align*}\label{5}
    |\langle f,g\rangle|\lesssim& \|g\|_{\mathcal C^{p(\cdot)}_D}\sum\limits_{\ell=-\infty}^\infty
    \sum\limits_{{\mathcal Q}^*_\ell\in B^*_\ell}\lambda_{\ell,\mathcal Q_\ell^\ast}\lesssim
\|g\|_{\mathcal C^{p(\cdot)}_D}\|f\|_{H_D^{p(\cdot)}} .
    \end{align*}
    
    The proof of Proposition \ref{holder} is complete.
    \end{proof}    

Now we state the main result in this section.
\begin{theorem}\label{dual-dunkl}
Suppose that $p(\cdot)\in LH$ with $\frac{N}{N+1}<p^-\le p^+\le 1$. The dual space of $H_D^{p(\cdot)}(\mathbb R^n)$ is $\mathcal C^{p(\cdot)}_D(\mathbb R^n)$ with in the following sense:
    \begin{enumerate}
        \item Let $g\in\mathcal C^{p(\cdot)}_D(\mathbb R^n)$. Then the linear functional 
        \[
        L_{g}\colon f\to L_{g}(f):=\int_{\mathbb R^{n}}f(x)g(x)d\omega(x),
        \]
        initially defined by any $f\in L^2(\mathbb R^n, d\omega)\cap H_D^{p(\cdot)}(\mathbb R^n)$, has a bounded extension to $H_D^{p(\cdot)}(\mathbb R^n)$.
        \item Conversely, for any continuous linear functional $L$ on $H_D^{p(\cdot)}(\mathbb R^n)$, there exists a unique function $g\in \mathcal C^{p(\cdot)}_D(\mathbb R^n)$ and $h:=(T_M^{-1})f$ such that $L(f)=\langle h, g\rangle$ holds for any $f\in H_D^{p(\cdot)}(\mathbb R^n)$ with $\|g\|_{\mathcal C^{p(\cdot)}_D}\lesssim \|L\|$.
    \end{enumerate}
\end{theorem}

\begin{proof}
First we show $(1)$. By the definition of $\mathcal C^{p(\cdot)}_D(\mathbb R^n)$, for any $g\in \mathcal C^{p(\cdot)}_D(\mathbb R^n)$,  there exists a sequence $g_n\in L^2(\mathbb R^n, d\omega)\cap \mathcal C^{p(\cdot)}_D(\mathbb R^n)$ with 
$\|g_n\|_{\mathcal C^{p(\cdot)}_D}\lesssim \|g\|_{\mathcal C^{p(\cdot)}_D}$ 
such that $\big<f, g\big>=\lim_{n\rightarrow\infty}\big<f, g_n\big>$, where $f\in L^2(\mathbb R^n, d\omega)\cap H_D^{p(\cdot)}(\mathbb R^n)$.
Applying Proposition \ref{holder}, we deduce that for $f\in L^2(\mathbb R^n, d\omega)\cap H_D^{p(\cdot)}(\mathbb R^n)$,
$$
|\big<f, g_n\big>|\lesssim \|f\|_{H_D^{p(\cdot)}}\|g_n\|_{\mathcal C^{p(\cdot)}_D}\lesssim  \|f\|_{H_D^{p(\cdot)}}\|g\|_{\mathcal C^{p(\cdot)}_D}.
$$
From the fact that $f\in L^2(\mathbb R^n, d\omega)\cap H_D^{p(\cdot)}(\mathbb R^n)$ is dense in $H_D^{p(\cdot)}(\mathbb R^n)$, 
we conclude that 
$g\in C_D^{p(\cdot)}(\mathbb R^n)$ can be extended to a continuous linear functional on $H_D^{p(\cdot)}(\mathbb R^n)$
 and $L_g:=\big<f, g\big>$ with
$$
\|L_g\|\lesssim \|g\|_{C_D^{p(\cdot)}},
$$
that is, 
$\mathcal C_D^{p(\cdot)}(\mathbb R^n)\subset (H_D^{p(\cdot)}(\mathbb R^n))'$.

To show $(2)$, we introduce the sequence spaces as follows.
For $p(\cdot)\in LH$, $0<p^-\le p^+\leq1$, the sequence space
$s^{p(\cdot)}$ consists all complex-value sequences
$$
s^{p(\cdot)}=\left\{
\{s_Q\}_Q:\|s_Q\|_{s^{p(\cdot)}}:=
\left\|\bigg\{\sum_Q
|s_Q|^2\chi_Q\bigg\}^{1/2}\right\|_{L^{p(\cdot)}}<\infty
\right\};
$$
The sequence space $c^{p(\cdot)}$ consists all complex-value sequences
$$
c^{p(\cdot)}=\left\{
\{t_Q\}_Q:\|t_Q\|_{c^{p(\cdot)}}:=
\sup_{P}\left\{
\frac{\omega(P)}{\|\chi_P\|^2_{L^{p(\cdot)}}}
\sum_{Q\subset P}\omega(Q)|t_Q|^2
\right\}^{1/2}
<\infty
\right\}.
$$

We now show that every continuous linear functional $l$ on $s^{p(\cdot)}$
satisfies $l=l_t$ for some $t_Q\in c^{p(\cdot)}$ with $\|t_Q\|_{c^{p(\cdot)}}
\le C\|l\|$.
For $s=\{s_Q\}_Q\in s^{p(\cdot)}$,
let $l(s)=\sum_{Q}\omega(Q)s_Qt_Q$.
Fix a dyadic cube $P$.
Let $X$ be the sequence space consisting of
$s=\{s_Q\}_{Q\subset P}$, and define a counting measure on dyadic cubes
$Q\subset P$ by
$d\sigma(Q)=\frac{\omega(Q)}{\omega(P)^{-1}\|\chi_P\|^2_{L^{p(\cdot)}}}$.

Then
\begin{align*}
&\left(\frac{\omega(P)}{\|\chi_P\|^2_{L^{p(\cdot)}}}
\sum_{Q\subset P}\omega(Q)|t_Q|^2\right)^{1/2}\\
&=\left\|\left\{t_Q\right\}_{Q\subset P}\right\|_{l^2(X,d\sigma)}\\
&=\sup_{\|s\|_{l^2(X,d\sigma)}\le 1}
\left|\frac{\omega(P)}{\|\chi_P\|^2_{L^{p(\cdot)}}}
\sum_{Q\subset P}\omega(Q)s_Qt_Q\right|\\
&=\sup_{\|s\|_{l^2(X,d\sigma)}\le 1}
\left|\ell\left(\frac{\omega(P)}{\|\chi_P\|^2_{L^{p(\cdot)}}}
s_Q\right)\right|\\
&=\|\ell\|\sup_{\|s\|_{l^2(X,d\sigma)}\le 1}
\left\|\left\{\frac{\omega(P)s_Q}{\|\chi_P\|^2_{L^{p(\cdot)}}}
\right\}_{Q\subset P}\right\|_{s^{p(\cdot)}}.
\end{align*}

For any $x\in \mathbb R^n$, choose that $0<r(x)<\infty$ such that
$\frac{1}{p(x)}=1+\frac{1}{r(x)}$.
By the generalized H\"older inequality in the variable setting, we have

\begin{align*}
&\left\|\left\{\frac{\omega(P)s_Q}{\|\chi_P\|^2_{L^{p(\cdot)}}}
\right\}_{Q\subset P}\right\|_{s^{p(\cdot)}}\\
&=
\left\|\left\{\sum_{Q\subset P}\frac{\omega(P)^2|s_Q|^2}
{\|\chi_P\|^4_{L^{p(\cdot)}}}\chi_Q
\right\}^{1/2}\right\|_{L^{p(\cdot)}}\\
&\le C
\left\|\left\{\sum_{Q\subset P}\frac{\omega(P)^2|s_Q|^2}
{\|\chi_P\|^4_{L^{p(\cdot)}}}\chi_Q
\right\}^{1/2}\right\|_{L^1}\|\chi_P\|_{L^{r(\cdot)}}\\
&=C\omega(P)\left\{\frac{1}{\omega(P)}\int_P
\bigg(\sum_{Q\subset P}\frac{|s_Q|^2}
{(\|\chi_P\|^2_{L^{p(\cdot)}}\omega(P)^{-1})^2}
\chi_Q(x)\bigg)^{1/2}d\omega(x)\right\}\|\chi_P\|_{L^{p(\cdot)}}\omega(P)^{-1}\\
&\le C\left\{\frac{1}{\omega(P)}\int_P
\bigg(\sum_{Q\subset P}\frac{|s_Q|^2}
{(\|\chi_P\|^2_{L^{p(\cdot)}}\omega(P)^{-1})^2}
\chi_Q(x)\bigg)d\omega(x)\right\}^{1/2}\|\chi_P\|_{L^{p(\cdot)}}\\
&\le C\left\{\int_P
\bigg(\sum_{Q\subset P}\frac{|s_Q|^2}
{\|\chi_P\|^2_{L^{p(\cdot)}}\omega(P)^{-1}}
\chi_Q(x)\bigg)d\omega(x)\right\}^{1/2}\\
&\le C
\bigg(\sum_{Q\subset P}\frac{|s_Q|^2\omega(Q)}
{\|\chi_P\|^2_{L^{p(\cdot)}}|P|^{-1}}
\bigg)^{1/2}\\
&=\|s\|_{l^2(X,d\sigma)}.
\end{align*}

Thus,
\begin{align*}
\|t_Q\|_{c^{p(\cdot)}}=\left(\frac{\omega(P)}{\|\chi_P\|^2_{L^{p(\cdot)}}}
\sum_{Q\subset P}|t_Q|^2\omega(Q)\right)^{1/2}
&\le\|\ell\|.
\end{align*}

Define a linear map $\mathcal L$ by
$$
\mathcal L (f)=\{q_Q(T_M)^{-1}f(x_Q)\}_Q,
$$
and another linear map $\mathcal P$ by
$$
\mathcal P(\{s_Q\}_Q)=\sum_{Q}\omega(Q)s_Q\psi_Q.
$$
By Theorem \ref{th1.1}, we find that $\mathcal P\circ \mathcal L$ is  the identity on $L^2(\mathbb R^n, d\omega)$.
Then let $L\in(H_D^{p(\cdot)}(\mathbb R^n))'$ and define $\ell=L\circ \mathcal P$ on $s^{p(\cdot)}$.
Observe that $$
\|\mathcal P(\{s_Q\})\|_{H_D^{p(\cdot)}}\lesssim \|S(\mathcal P(\{s_Q\}))\|_{L^{p(\cdot)}}
\lesssim \|\{s_Q\}\|_{s_{p(\cdot)}}.
$$

Hence, we conclude that $\ell\in (s^{p(\cdot)})'$.
Thus, there exists $t=\{t_Q\}_Q\in c^{p(\cdot)}$ such that
$$
\ell(\{s_Q\}_Q)=\sum_Q\omega(Q)s_Qt_Q
$$
and
$\|t\|_{c^{p(\cdot)}}\lesssim\|\ell\|\lesssim \|L\|$.

For $f\in H_D^{p(\cdot)}(\mathbb R^n)$, we deduce that
\begin{align*}
L(f)&=L(\sum_j\sum_Q\omega(Q)q_Q(T_M^{-1})f(x_Q)\psi_Q(x,x_Q))\\
&=L\circ \mathcal P\circ \mathcal L(\sum_j\sum_Q\omega(Q)q_Q(T_M^{-1})f(x_Q)\psi_Q(x,x_Q))\\
&=\ell(\{q_Q(T_M^{-1})f(x_Q)\})\\
&=\sum_j\sum_Q \omega(Q)q_Q(T_M^{-1})f(x_Q)t_Q\\
&=\sum_j\sum_Q \omega(Q)\big<q_Q(\cdot, x_Q), (T_M^{-1})f(\cdot)\big>t_Q\\
&=\big<(T_M^{-1})f(\cdot), \sum_j\sum_Q \omega(Q)q_Q(\cdot, x_Q)t_Q\big>\\
&=:\big<(h,g\big>,
\end{align*}
where $g(x):= \sum_j\sum_Q \omega(Q)q_Q(x, x_Q)t_Q$
and $h:=(T_M^{-1})f$ with $\|h\|_{H_D^{p(\cdot)}}\sim \|f\|_{H_D^{p(\cdot)}}$.

Therefore,
$$
\|g\|_{\mathcal C_D^{p(\cdot)}}\lesssim \|t_Q\|_{c^{p(\cdot)}}
\le C\|L\|
$$
and the proof is complete.

\end{proof}

\section{{$T1$ Theorem for variable Dunkl--Hardy and Dunkl--Carleson spaces}}

In this section, 
we will obtain the general $T1$ theorem for variable Dunkl--Hardy spaces and 
variable Dunkl--Carleson spaces by using almost orthogonality estimates together with a weak density argument.

Before we state the main results in this section, we recall the definition and some boundedness results of Dunkl--Calder\'on--Zygmund operator.
  Let ${\dot{C}}^\eta(\Bbb
  R^n)$ be the H\"older space of continuous functions $f$ with 
  $$\|f\|_{{\dot{C}}^\eta}:=\sup\limits_{x\ne y} \frac{|f(x)-f(y)|}{\|x-y\|^{\eta}}<\infty. $$
  We denote ${\dot{C}}^\eta_0(\Bbb R^n)$ by the H\"older space ${\dot{C}}^\eta(\Bbb R^n)$ with compact supports.   
Now we recall the following definition of Dunkl--Calder\'on--Zygmund operator which is introduced in \cite{THHLL1}.
\begin{definition}\label{df1.12}
        An operator $T: C_0^\eta(\mathbb{R}^n)\rightarrow(C_0^\eta(\mathbb{R}^n))'$ for some $\eta>0,$ is said to be a Dunkl--Calder\'on--Zygmund singular integral operator if $K(x,y),$ the kernel of $T$,
        satisfies the following conditions: there exists some $0<\varepsilon\leqslant 1$ such that
    \begin{equation}\label{si}
    |K(x,y)|\lesssim \frac1{\omega(B(x,d(x,y)))}\Big(\frac{d(x,y)}{\|x-y\|}\Big)^\varepsilon
    \end{equation}
    for all ${  d(x,y)\not=0};$
    \begin{equation}\label{smooth y3}
    |K(x,y)-K(x,y')|\lesssim \Big(\frac{\|y-y'\|}{\|x-y\|}\Big)^\varepsilon\frac{1}{\omega(B(x,d(x,y)))}
    \end{equation}
    for $\|y-y'\|\leqslant  d(x,y)/2;$
    \begin{equation}\label{smooth x3}
    |K(x',y)-K(x,y)|\lesssim \Big(\frac{\|x-x'\|}{\|x-y\|}\Big)^\varepsilon\frac1{\omega(B(x,d(x,y)))}
    \end{equation}
    for $\|x-x'\|\leqslant  d(x,y)/2.$

    Moreover,
    $$\langle T(f),g\rangle=\int_{\R^N}\int_{\R^n} K(x,y)f(x)g(y)d\omega(x)d\omega(y)$$
    for all $f$ and $g$ in $C_0^\eta(\mathbb{R}^n)$ with $\supp f\cap \supp g=\emptyset.$

    A Dunkl--Calder\'on--Zygmund singular integral operator is said to be the Dunkl--Calder\'on--Zygmund operator if it extends a bounded operator on $L^2(\mathbb R^n,\omega).$
\end{definition}

The boundedness of the Dunkl-Calder\'on-Zygmund operators on the
variable Dunkl--Lebesgue spaces is given in \cite{TT}. 
  \begin{theorem}\label{bdlp}
  	Suppose that $T$ is a Dunkl--Calder\'on--Zygmund operator and  $p(\cdot)\in  \mathcal P\cap LH$  with $\frac{{N}}{{N}+\epsilon}<p^-\le p^+<\infty$. Then $T$ extends to a bounded operator on the variable Dunkl--Lebesgue space $L^{p(\cdot)}(\mathbb R^n, dw)$. 
Moreover, there is a constant $C$ such that
  	$$\|T(f)\|_{L^{p(\cdot)}(\mathbb R^n, dw)} \leqslant C\|f\|_{L^{p(\cdot)}(\mathbb R^n, dw)}.$$ 	
  \end{theorem}
  
Furthermore, \cite{Tan26-1} gives the boundedness of the Dunkl--Calder\'on--Zygmund operators on the
variable Dunkl--Hardy spaces.
   \begin{theorem}\label{bdhp}
  	Suppose that $T$ is a Dunkl--Calder\'on--Zygmund operator and  $p(\cdot)\in\mathcal P^0\cap LH$ with $\frac{{N}}{{N}+\epsilon}<p^-\le p^+<\infty$. 
Then $T$ extends to a bounded operator on the variable Dunkl--Hardy space $H_D^{p(\cdot)}(\mathbb R^n)$ to $L^{p(\cdot)}(\mathbb R^n,dw)$. 
Moreover, there is a constant $C$ such that
  	$$\|T(f)\|_{L^{p(\cdot)}(\mathbb R^n,dw)} \leqslant C\|f\|_{H_D^{p(\cdot)}(\mathbb R^n)}.$$ 	
  \end{theorem}
  
  Besides, in the proof of the main result in this section, we also need the definition of the paraproduct operator as follows.

\begin{definition}\label{def:paraproduct}
Suppose that $\{S_k\}$, $\{D_k\}$ and $\{\widetilde{\widetilde{D}}_k\}$ are the same as defined above. The paraproduct operator of $f\in\mathcal{M}(\beta,\gamma,r,x_0)'$ is defined by
\begin{align*}
\Pi_b(f)(x)
&=
\sum_{k=-\infty}^{\infty}
\sum_{Q\in\mathcal{Q}^k}
\omega(Q)
D_k(x,x_Q)
\widetilde{\widetilde{D}}_k(b)(x_Q)
S_k(f)(x_Q),
\end{align*}
where $b\in BMO(\mathbb{R}^n,d\omega)$.
\end{definition}

\begin{remark}\label{rem:paraproduct}
By \cite[Theorem 2.15]{THHLL1}, the paraproduct operator $\pi_b$ is the Calder\'on--Zygmund operator.
By the classical method, we find that $\pi_b$ is bounded on the variable Hardy space $H^{p(\cdot)}_{cw}(\mathbb R^n,d\omega)$
for $p(\cdot)\in LH$ with $\frac{{N}}{{N}+\epsilon}<p^-\le p^+\le 1$.
See for example \cite{Tan2020-para}.
\end{remark}
 
Now we are in a position to state the main result in this section. 
\begin{theorem}\label{t1}
    Suppose that $T$ is a Dunkl--Calder\'on--Zygmund operator with the exponent of the regularity of the kernel $\varepsilon$ and  $p(\cdot)\in\mathcal P^0\cap LH$ with $\frac{{N}}{{N}+\epsilon}<p^-\le p^+<\infty$.      
    Then $T$ is bounded on $H_D^{p(\cdot)}(\mathbb R^n)$ if and only if $T^*(1)=0$. Moreover, $T$ is also bounded on 
    $\mathcal C_D^{p(\cdot)}(\mathbb R^n)$ if and only if $T(1)=0.$
\end{theorem}

\begin{proof}
If $T$ is bounded on $H_D^{p(\cdot)}(\mathbb R^n)$ for $p(\cdot)\in\mathcal P^0\cap LH$ with $\frac{{N}}{{N}+\epsilon}<p^-\le p^+<\infty$,
then it is also bounded on $H_D^{p_0}(\mathbb R^n)$ when $p(\cdot)=p_0\in (\frac{N}{N+\epsilon}, \infty)$.
By \cite[Theorem 1. 27]{THHLL2}, we obtain that, in this case, $T^\ast(1)=0$. 

Now we prove that if $T$ is a Dunkl--Calder\'on--Zygmund operator with $T^*(1)=0,$ then $T$ is bounded on the Dunkl--Hardy space 
$H_D^{p(\cdot)}(\mathbb R^n,d\omega)$ for $p(\cdot)\in\mathcal P^0\cap LH$ with $\frac{{N}}{{N}+\epsilon}<p^-\le p^+<\infty$.
We first assume that $T$ fulfills $T^\ast(1)=T(1)=0.$ 
By the fact that $L^2(\mathbb R^n, d\omega)\cap H_D^{p(\cdot)}(\mathbb R^n)$ is dense in $H_D^{p(\cdot)}(\mathbb R^n),$ we only need to prove 
$$\|T(f)\|_{H_D^{p(\cdot)}}\leqslant C\|f\|_{H_D^{p(\cdot)}}$$
    for $f\in L^2(\R^N,\omega)\cap H_d^p(\R^N,\omega).$ 
By Theorem \ref{th1.1}, we know that, if $f\in L^2(\mathbb R^n,d\omega)\cap H_D^{p(\cdot)}(\mathbb R^n),$ then 
$$f(x)=\sum\limits_{k'=-\infty}^\infty\sum\limits_{Q'\in {\mathcal Q'}^{k'}}w(Q')q_{Q'}h(x_{Q'})\psi_{Q'}(x,x_{Q'}),$$ 
where the series converges in both $L^2(\mathbb R^n, d\omega)$ and $H_D^{p(\cdot)}(\mathbb R^n).$
    We have
    \begin{align*}
        &\|T(f)\|_{H_D^{p(\cdot)}}=\|S\big(T(f)\big)\|_{L^{p(\cdot)}}\\
        &=\left\|\left\{\sum\limits_{k=-\infty}^\infty\sum\limits_{Q\in {\mathcal Q}^k}\left|q_Q\Big(\sum\limits_{k'=-\infty}^\infty\sum\limits_{Q'\in {\mathcal Q}^{k'}}\omega(Q')q_{Q'}h(x_{Q'})T(\psi_{Q'}(\cdot,x_{Q'})\Big)(x_Q)\right|^2
        \chi_{Q}\right\}^{\frac{1}{2}}\right\|_{L^{p(\dot)}}\\
        &=\left\|\left\{\sum\limits_{k=-\infty}^\infty\sum\limits_{Q\in {\mathcal Q}^k}\left|\sum\limits_{k'=-\infty}^\infty\sum\limits_{
        Q'\in {\mathcal Q}^{k'}}\omega(Q')q_{Q'}h(x_{Q'})q_QT\psi_{Q'}(x_Q,x_{Q'})\right|^2\chi_{Q}\right\}^{\frac{1}{2}}\right\|_{L^{p(\cdot)}}.
    \end{align*}


By using Lemma \cite[Lemma 2.10]{THHLL2}, we deduce that
     \begin{align*}
    S_{k',k}(x_Q,x_Q')|&:=|q_QT\psi_{Q'}(x_Q,x_Q')|\\
    &\lesssim r^{-|k-k'|\varepsilon'}
    \frac1{V(x,y, r^{-{k'}\vee -k}+d(x,y))}\Big(\frac{r^{-{k'}\vee -k}}{r^{-{k'}\vee -k}+d(x,y)}\Big)^\gamma,
     \end{align*}
    where $\gamma, \epsilon'\in (0, \epsilon)$.
    Then for $p(\cdot)\in\mathcal P^0\cap LH$ with $\frac{{ N}}{{N}+\epsilon}<p^-\le p^+<\infty$,
        \begin{align*}
        &\|T(f)\|_{H_D^{p(\cdot)}}\\
        &\lesssim
        \left\|\bigg(\sum\limits_{k'=-\infty}^\infty\sum\limits_{Q'\in {\mathcal Q}^{k'}}
        \Big|\sum\limits_{k=-\infty}^\infty\sum\limits_{Q\in {\mathcal Q}^k} \omega(Q)S_{k',k}(x_{Q'},x_Q)q_Qh(x_Q)\Big|^2\chi_{Q'}\bigg)^{\frac{1}{2}}\right\|_{L^{p(\cdot)}(\mathbb R^n,dw)} \\
        &\lesssim
        \left\| \bigg(\sum\limits_{k=-\infty}^\infty\sum\limits_{Q\in {\mathcal Q}^k} |q_Qh(x_Q)|^2\chi_{Q}\bigg)^{\frac{1}{2}}\right\|_{L^{p(\cdot)}(\mathbb R^n,dw)}\lesssim \|h\|_{H_D^{p(\cdot)}}\sim \|f\|_{H_D^{p(\cdot)}}.
        \end{align*}
        
        Indeed, by the  \begin{align*}
        &\sum\limits_{k'=-\infty}^\infty\sum\limits_{Q'\in {\mathcal Q}^{k'}}
        \Big|\sum\limits_{k=-\infty}^\infty\sum\limits_{Q\in {\mathcal Q}^k} \omega(Q)S_{k',k}(x_{Q'},x_Q)q_Qh(x_Q)\Big|^2\chi_{Q'} \\
        &\lesssim \sum\limits_{\sigma\in G}\sum\limits_{k=-\infty}^\infty
        \bigg\{M\Big(\sum\limits_{Q\in {\mathcal Q}^k}|q_Qh(x_Q)|^\theta\chi_{Q}\Big)(\sigma(x))\bigg\}^{2/\theta},
        \end{align*}
        where $\sup\limits _{k\in \Bbb Z}\sum\limits_{k'=-\infty}^\infty r^{-|k-k'|\varepsilon}r^{[-k-(-k'\vee -k)]{ N}(1-\frac 1\theta)}<\infty$ for $\frac{N}{{ N}+\varepsilon}<\theta<p_-\leqslant 1.$
Finally, by Lemma \ref{FS}, we conclude that
        \begin{align*}
        &\left\| \bigg(\sum\limits_{k'=-\infty}^\infty\sum\limits_{Q'\in {\mathcal Q}^{k'}}
        \Big|\sum\limits_{k=-\infty}^\infty\sum\limits_{Q\in {\mathcal Q}^k} \omega(Q)S_{k,k'}(x_{Q'},x_Q)q_Qh(x_Q)\Big|^2\chi_{Q'}\bigg)^{\frac{1}{2}}\right\|_{L^{p(\cdot)}(\mathbb R^n,dw)}  \\
        \lesssim& \sum\limits_{\sigma\in G}
        \left\|\Big(\sum\limits_{k=-\infty}^\infty\sum\limits_{Q\in {\mathcal Q}^k}|q_Qh(x_Q)|^2\chi_Q(\sigma(x))\Big)^{\frac{1}{2}}\right\|_{L^{p(\cdot)}(\mathbb R^n,dw)}\\
        \lesssim&
        \left\|\Big(\sum\limits_{k=-\infty}^\infty\sum\limits_{Q\in {\mathcal Q}^k}|q_Qh(x_Q)|^2\chi_Q(x)\Big)^{\frac{1}{2}}\right\|_{L^{p(\cdot)}(\mathbb R^n,dw)}.
        \end{align*}

Denote $ {\widetilde T}= T- \Pi_{T(1)}.$ We find that ${\widetilde T}(1)={(\widetilde T)^*}(1)=0.$ 
Hence, ${\widetilde T}$ is bounded on $H_D^{p(\cdot)}(\mathbb R^n)$ for $p(\cdot)\in LH$ with $\frac{{N}}{{N}+\epsilon}<p^-\le p^+\le 1$
Form Remark \ref{rem:paraproduct}, $\Pi_{T(1)}$ is bounded on the variable Hardy space $H_{cw}^{p(\cdot)}(\mathbb R^n)$.
Therefore, we conclude that $T$ is bounded on $H_D^{p(\cdot)}(\mathbb R^n)$ for $p(\cdot)\in LH$ with $\frac{{N}}{{N}+\epsilon}<p^-\le p^+\le 1$.


On the other hand, if $T$ is also bounded on 
    $\mathcal C_D^{p(\cdot)}(\mathbb R^n)$ for $p(\cdot)\in\mathcal P^0\cap LH$ with $\frac{{N}}{{N}+\epsilon}<p^-\le p^+<\infty$,
    then $$\|T(1)\|_{\mathcal C_D^{p(\cdot)}}\lesssim \|1\|_{\mathcal C_D^{p(\cdot)}}.$$
By the fact that $\|1\|_{\mathcal C_D^{p(\cdot)}}=0$, in this case we obtain that $T(1)=0.$
In order to complete the proof of this theorem, we only need to
show the boundedness of $T$ on $\mathcal C_D^{p(\cdot)}(\mathbb R^n)$. 
By using Theorem \ref{dual-dunkl},
we obtain that
$$
|\left<Tf,g\right>|=|\left<f,T^\ast g\right>|\le
\|f\|_{\mathcal C_D^{p(\cdot)}}\|T^\ast g\|_{H_D^{p(\cdot)}}
\le C\|f\|_{\mathcal C_D^{p(\cdot)}}\|g\|_{H_D^{p(\cdot)}}.
$$
Thus, for each $f\in \mathcal C_D^{p(\cdot)}(\mathbb R^n)\cap L^2(\mathbb R^n, d\omega)$,
$L_f(g)=\left<Tf,g\right>$ is a continuous linear functional
on $H_D^{p(\cdot)}(\mathbb R^n)\cap L^2(\mathbb R^n, d\omega)$.
Since $H_D^{p(\cdot)}(\mathbb R^n)\cap L^2(\mathbb R^n, d\omega)$ is dense in $H_D^{p(\cdot)}(\mathbb R^n)$,
$L_f$
can be extended to a continuous linear functional
on $H_D^{p(\cdot)}(\mathbb R^n)$ with
$$\|L_f\|\le C\|f\|_{\mathcal C_D^{p(\cdot)}}.$$
Applying Theorem \ref{dual-dunkl} again, there exists
$h\in \mathcal C_D^{p(\cdot)}(\mathbb R^n)$ such that $\left<Tf,g\right>
=\left<h,g\right>$ for $g\in H_D^{p(\cdot)}(\mathbb R^n)\cap L^2(\mathbb R^n,d\omega)$ with
$$\|h\|_{\mathcal C_D^{p(\cdot)}}\le C\|L_f\|.$$

Thus,
\begin{align*}
\|Tf\|_{CMO^{p(\cdot)}}
&= \sup\limits _P \left\{\frac{\omega(P)}{\|\chi_P\|^2_{L^{p(\cdot)}}}\sum\limits_{Q\subseteq P} \omega(Q) |\psi_Q(T(f)(\cdot,y))(x_Q)|^2
\right\}^{\frac{1}{2}}\\
        &=\sup\limits _P  \left\{\frac{\omega(P)}{\|\chi_P\|^2_{L^{p(\cdot)}}}\sum\limits_{Q\subseteq P} \omega(Q) |\psi_Q(h)(\cdot,y))(x_Q)|^2
        \right\}^{\frac{1}{2}}\\
&=\|h\|_{\mathcal C_D^{p(\cdot)}}\lesssim \|L_f\|\lesssim \|f\|_{\mathcal C_D^{p(\cdot)}}.
\end{align*}

Now we extend this result to the whole space $\mathcal C_D^{p(\cdot)}(\mathbb R^n)$.
For any $f\in \mathcal C_D^{p(\cdot)}(\mathbb R^n)$, then there is a sequence 
$\lbrace f_n\rbrace\subseteq L^2(\mathbb R^n,d\omega)\cap \mathcal C_D^{p(\cdot)}(\mathbb R^n)$ 
such that $ \Vert f_n\Vert_{ \mathcal C_D^{p(\cdot)}}\lesssim \Vert f\Vert_{ \mathcal C_D^{p(\cdot)}},$
     and  for each $g\in L^2(\mathbb R^n,d\omega)\cap H_D^{p(\cdot)}(\mathbb R^n),$ 
     $\langle f_n, g\rangle\rightarrow \langle f, g\rangle$ as $n\rightarrow \infty.$ 
     Therefore, for  $f\in  \mathcal C_D^{p(\cdot)}(\mathbb R^n)$, by the duality we define
    $$ \langle T(f), g\rangle:=\lim \limits_{n \rightarrow \infty} \langle T(f_n), g\rangle$$
    for each $g\in L^2(\R^N,\omega)\cap H_d^p(\R^N,\omega).$
    
In fact, $\langle T(f_j-f_k), g\rangle=\langle f_j-f_k, T^*(g)\rangle$,
where $f_j-f_k$ and $g $ belong to $L^2$ and $T$ is bounded on $L^2$. 
Hence, $T^*$ is bounded on $L^2$ and the kernel of $T^*$ satisfies the conditions  $T(1)=(T^*)^*(1)=0.$
Then by the boundedness on $H_D^{p(\cdot)}$ above, $T^*(g) \in L^2(\mathbb R^n,d\omega)\cap  H_D^{p(\cdot)}(\mathbb R^n).$ 
By Proposition  \ref{holder}, $\langle f_j-f_k, T^*(g)\rangle$ tends to zero as $j,k\rightarrow \infty.$
Therefore, $Tf$ is well defined.

    To finish the proof of ``if'' part, we claim that for each $f\in L^2(\R^N,\omega)\cap  CMO_d^p(\R^N,\omega)$,
    \begin{eqnarray}\label{T bd on L2 cap CMO}
    \|T(f)\|_{CMO_d^p(\R^N,\omega)}\leqslant C\|f\|_{CMO_d^p(\R^N,\omega)},
    \end{eqnarray}
    where the constant $C$ is independent of $f$.

Choose $g=\psi_Q$. Then using Fatou’s Lemma, we deduce that
    \begin{align*}
        \|T(f)\|_{{\mathcal C_D^{p(\cdot)}}}&=\| \lim_{n\rightarrow\infty}T(f_n)\|_{{\mathcal C_D^{p(\cdot)}}}
        \leqslant  \liminf_{n\rightarrow\infty} \|T(f_n)\|_{{\mathcal C_D^{p(\cdot)}}}\\
        &\lesssim  \|f_n\|_{{\mathcal C_D^{p(\cdot)}}}\lesssim  \|f\|_{{\mathcal C_D^{p(\cdot)}}},
    \end{align*}
which completes this proof.

 \end{proof}

{\bf Acknowledgments.}




The project is sponsored by
the National Natural Science Foundation of China(Grant No. 11901309) and 
the Open Project
Program of Key Laboratory of Mathematics and Complex System (Grant No. K202502), Beijing Normal University.
Part of this work was completed during the author's visit to Auburn University. The author would like to thank Auburn University for its hospitality.

\bigskip
\bigskip

\medskip
\noindent Jian Tan\\
\noindent School of Science,\\
Nanjing University of Posts and Telecommunications,\\
Nanjing 210023, People's Republic of China\\

\noindent {\it E-mail address}: \texttt{tj@njupt.edu.cn}(J. Tan)\\

\end{document}